\documentclass[10pt]{article}
  \usepackage{amsmath, amssymb, amsthm}
  \usepackage{graphicx,float,mathrsfs,enumitem}
  \usepackage{xpatch,thmtools,wasysym,algorithm,algpseudocode}
  \usepackage{hyperref,cleveref}
  \usepackage[margin=1in]{geometry}
  \usepackage{units}[loose,nice]
  \usepackage{multirow}
  \usepackage{booktabs}
  \usepackage[toc]{appendix}
  \usepackage{xcolor,soul}
  \usepackage{bm}
  
  \usepackage{setspace}

  \newtheorem{assumption}{Assumption}[section]

  \newtheorem{theorem}{Theorem}[section]
  \newtheorem{lemma}{Lemma}[section]
  
  \newtheorem{corollary}{Corollary}[section]

  \DeclareMathOperator{\boldsymbol{u}}{\boldsymbol{u}}

\title{Anderson Acceleration For Perturbed Newton Methods}
\author{Matt Dallas \footnote{
mdallas@udallas.edu 
} %$^{,1}$ 
\\
Department of Mathematics, University of Dallas
}
\date{\today}

\begin{document}
\maketitle

%%%%%%%%%%%%%%%%%%%%%%%%%%%%%%%%%%%%%%%%%%%%%%%%%%%%%%%
%             5. ABSTRACT
%%%%%%%%%%%%%%%%%%%%%%%%%%%%%%%%%%%%%%%%%%%%%%%%%%%%%%%

\abstract{
We present a convergence theory for Anderson acceleration (AA) applied to perturbed Newton methods (pNMs) for computing roots of nonlinear problems. Two important special cases are the classical Newton method and the Levenberg-Marquardt method. We prove that if a problem is 2-regular, then Anderson accelerated pNMs coupled with a safeguarding scheme, known as $\gamma$-safeguarding, converge locally linearly in a starlike domain of convergence, but with an improved rate of convergence compared to standard perturbed Newton methods. Since Levenberg-Marquardt methods are a special case of pNMs, we obtain a novel acceleration and local convergence result for Anderson accelerated Levenberg-Marquardt. We further show that  $\gamma$-safeguarding can detect if the underlying perturbed Newton method is converging superlinearly, and respond by tuning the Anderson step down. We demonstrate the methods on several benchmark problems in the literature.
}

%%%%%%%%%%%%%%%%%%%%%%%%%%%%%%%%%%%%%%%%%%%%%%%%%%%%%%
% 6. BODY
%%%%%%%%%%%%%%%%%%%%%%%%%%%%%%%%%%%%%%%%%%%%%%%%%%%%%%

\section{Introduction} 

In virtually every branch of science, one encounters the problem of finding a root $x^*$ of a nonlinear function $f:\mathbb{R}^n\to\mathbb{R}^n$. Solving these problems efficiently and accurately remains an active area of study. Several popular nonlinear solvers such as standard Newton \cite{deuf05}, Levenberg-Marquardt \cite{FiIzSo24}, inexact Newton \cite{EisWa96}, and regularized Newton \cite{pollock_regularized_newton}, are particular instances of {\it perturbed Newton methods}, or pNMs \cite{IzSo25}. Two special cases of interest in this work are the standard Newton method and the Levenberg-Marquardt method (LM). Each of these solvers converge locally quadratically to $x^*$ when the Jacobian $f'(x^*)$ is nonsingular \cite{deuf05,KaYaFu04}. Moreover, LM converges locally quadratically under a weaker condition known as the local-error bound condition \cite{KaYaFu04}. On the other hand, if $f'(x^*)$ is singular and 2-regular (see Section \ref{sec:pnm}), then Newton and LM converge locally linearly to $x^*$ within a starlike domain \cite{Gr80,IzKuSo18-1}. It is shown in \cite[Proposition 1]{IzKuSo18-1} that the local error bound condition does not hold if $f$ is 2-regular under mild conditions on $f^{-1}(0)$. See the recent survey paper \cite{FiIzSo24} for a more detailed account of the current LM theory. Whenever $f'(x^*)$ is singular, we say that $x^*$ is a {\it singular point}, and the corresponding root-finding problem $f(x)=0$ is a {\it singular problem}. These arise naturally in parameter-dependent problems where any bifurcation point is necessarily a singular point \cite{Ke18}. Even if the problem is not exactly singular, but nearly singular, Newton may only converge locally linearly \cite{DeKe85}. This motivates the study of accelerated Newton method and pNMs. 

The problem of accelerating Newton's method for singular problems has been studied extensively, see, for example, \cite{DaPo23,DaPoRe24,DeKeKe83,FiIzSo21-2,FiIzSo21-1,Gr80,Gr85,IzKuSo18-1,IzKuSo18-2} and references therein. In this article, we study a particular acceleration technique, called Anderson acceleration. It was invented in 1965 by D.G. Anderson to solve integral equations \cite{Anderson65}, and the decades since its introduction, the last 15 years in particular, have seen a significant body of literature develop on the theory and implementation of Anderson's method \cite{EPRX19,Eyert96,FaSa09,PoRe21,PoRe23,ReXi23,WaNi11}. It has also shown itself to be a valuable tool in a wide variety of applications \cite{AJW17,LWWY12,PaMa22,PRX18,sim-app-2018,TRL04,WHdS21,Yang21}, with new applications continuing to be explored, such as equilibrium chemistry \cite{chem_equilibrium_awada,chem_equilibrium_nazeer}. Anderson acceleration (AA) is an extrapolation scheme that can be applied to any fixed-point iteration by taking a linear combination of the previous $m+1$ iterates. The weights of the linear combination are chosen so as to minimize the corresponding linear combination of the previous $m+1$ residuals. The number $m$ is called the algorithmic depth, and for many applications, such as fluid problems, it is chosen to be no greater than 5, but it is occasionally set larger. 

The first main result of this paper, Theorem \ref{thm:main}, extends the theory developed in \cite{DaPo23,DaPoRe24} for Anderson accelerated Newton to Anderson accelerated pNMs (AApNMs). It is a  local convergence and acceleration result for a safeguarded version of AApNM, and we will show that this safeguarded version of AApNM does not reduce the order of convergence of the underlying pNM when the latter is superlinearly convergent. This is significant since Anderson acceleration is known to decrease the order of convergence of superlinearly convergent fixed-point iterations \cite{ReXi23}. As a corollary, we will obtain theoretical results for Anderson accelerated LM, or AALM. 
To the best of our knowledge, the results presented in this work are the first of their kind for Anderson acceleration applied to LM in the context of singular problems. 

The rest of the article proceeds as follows. Section \ref{sec:pnm} discusses pNMs and 2-regularity in more detail. Section \ref{sec:aa_sg} reviews Anderson acceleration, and recalls theory developed for the safeguarding method used in this work: $\gamma$-safeguarding. The analysis of $\gamma$-safeguarded AApNMs for both singular and nonsingular problems is presented in Section \ref{sec:cvg_analysis}. We conclude with selected numerical experiments in Section \ref{sec:numerics}. The experiments highlight the special cases of Anderson accelerated Newton's method, (AANewt), Anderson accelerated inexact Newton (AAinNewt), and Anderson accelerated Levenberg-Marquardt (AALM). 

The following notation will be used throughout the paper. Let $f:\mathbb{R}^n\to\mathbb{R}^n$ be a nonlinear, $C^2$ function. Let $f'(x)$ denote the Jacobian of $f$ at $x$. Unless stated otherwise, as in Section \ref{subsec:suplin-cvg}, we assume that $f'(x^*)$ is singular with null space $N$ and $Q=N^{\perp}$ is the orthogonal complement of $N$. Let $P_N$ and $P_Q$ denote the orthogonal projections onto these spaces. Letting $R$ denote the range space of $f'(x^*)$, we may assume that $Q = R$ since, like the standard Newton method, pNMs are essentially invariant under affine transformations of the domain and range (see Section \ref{sec:pnm}). Hence $P_Q=P_R$. In this work $\|\cdot\|$ always denotes the Euclidean 2-norm. Given a sequence $\{x_k\}$, we will write $e_k := x_k-x^*$ to denote the error at step $k$.

\section{Perturbed Newton Methods}
\label{sec:pnm}

The classical Newton method is defined as follows. \\
\begin{minipage}{1.0\linewidth}
\begin{algorithm}[H]
\begin{algorithmic}[1]
\caption{Newton (Newt)}
\label{alg:newt}
\State{Choose $x_0\in\mathbb{R}^n$.}
\For{$k=0,1,...$}
\State $w_{k+1}\gets -f'(x_k)^{-1}f(x_k)$
\State$x_{k+1}\gets x_k+w_{k+1}$
\EndFor
\end{algorithmic}
\end{algorithm}
\end{minipage}
\vspace{0.5em}

A perturbed Newton method, pNM, modifies Algorithm \ref{alg:newt} by defining the update step $w_{k+1}$ as the solution to 
\begin{align}
    \label{eq:pnm_update}
    [f'(x_k)+ \Omega_k]d = -f(x_k)+\omega_k
\end{align}
where $\Omega_k$ and $\omega_k$ are perturbation terms. We can then define the perturbed Newton algorithm.  \\
\begin{minipage}{1.0\linewidth}
\begin{algorithm}[H]
\begin{algorithmic}[1]
\caption{perturbed Newton (pNM)}
\label{alg:pnm}
\State{Choose $x_0\in\mathbb{R}^n$, $\Omega_0\in\mathbb{R}^{n\times n}$, and $\omega_0\in\mathbb{R}^n$}
\For{$k=0,1,...$}
\State $w_{k+1}\gets -(f'(x_k)+\Omega_k)^{-1}(f(x_k)-\omega_k)$
\State $x_{k+1}\gets x_k+w_{k+1}$
\State Update $\Omega_k$ and $\omega_k$
\EndFor
\end{algorithmic}
\end{algorithm}
\end{minipage}
\vspace{0.5em}
  
Like standard Newton, pNMs are essentially invariant under affine transformations. To clarify this statement, suppose we are given $F:
\mathbb{R}^n\to\mathbb{R}^n$ and $y^*$ such that $F(y^*)=0$. We can then define $f(x) = U^TF(Vx)$, where $y^* = Vx^*$, and $U$ and $V$ come from the SVD $F'(y^*) = U\Sigma V^T$. It follows that $f(x^*)=0$ and $f'(x^*) = U^T F'(y^*)V = \Sigma$. Thus $R = N^{\perp}$ for $f'(x^*)$. Moreover, the pNM sequence given by 
\begin{align}
    \label{eq:pnm_step}
    x_{k+1} = x_k - [f'(x_k) + \Omega_k]^{-1}(f(x_k)- \omega_k)
\end{align}
may be transformed into 
\begin{align}
    y_{k+1} = y_k - [F'(y_k)+U\Omega_kV^T]^{-1}(F(y_k)-U\omega_k)
\end{align}
by an orthogonal transformation, i.e., $y_{k} = Vx_{k}$ for all $k$. We also remark that, since $U$ and $V$ are orthogonal, $\|\Omega_k\| = \|U\Omega_kV^T\|$ and $\|\omega_k\|=\|U\omega_k\|$. It therefore suffices to analyze pNM applied to $f(x)$. Several popular iterative methods fall under the perturbed Newton framework. For instance, if $\Omega_k = 0$, and $\omega_k \neq 0$, we obtain an inexact Newton method. If  $(f'(x_k))^T$ is invertible, we can set $\omega_k=0$ and $\Omega_k = \mu_k(f'(x_k))^{-T}$ to obtain a Levenberg-Marquardt update with parameter $\mu_k$. Further, any quasi-Newton method where the update step $w_{k+1}$ solves $M_kd+f_k = 0$
may be cast as a pNM by selecting $\Omega_k = M_k-f'_k$ and $\omega_k =0$. The classical Newton method is obtained when $\Omega_k=0$ and $\omega_k =0$. 

The goal of this work is to analyze Anderson accelerated pNM for both singular and nonsingular problems. For singular problems, the key assumption is that $f(x)$ is {\it 2-regular} at $x^*$ along some vector $v^*\in N$. In general, $f(x)$ is 2-regular at $x^*$ along $v^*$ if the linear operator $f'(x^*)(\cdot) + P_{N}f''(x^*)(v^*,(\cdot))$ is invertible as a map from $N$ to $N$. This is equivalent to assuming that $T(\cdot) = P_{N}f''(x^*)(v^*,P_N(\cdot))$ is invertible as a map from $N$ to $N$, which is the main assumption used in \cite{DaPo23} in the analysis of Anderson acceleration applied to the classical Newton method. The following lemma from \cite{IzKuSo18-1} says that pNMs behave similarly to standard Newton when $f$ is 2-regular. 
\begin{lemma}
\label{lem:izmailov1}
{\it \cite[Lemma 1]{IzKuSo18-1}}
    Suppose that $f$ is $C^2$ and $f''(x)-f''(x^*) = \mathcal{O}(\|x-x^*\|)$ as $x\to x^*$, and let $f(x)$ be 2-regular at $x^*$ along some some vector $v^*$ of unit length. Let $\Omega(x) = \mathcal{O}(\|x-x^*\|)$ as $x\to x^*$, and suppose that for every $\Delta>0$, there is a $\varepsilon>0$ and $\delta>0$ such that if $\|x-x^*\|<\varepsilon$ and $\|\frac{x-x^*}{\|x-x^*\|}-v^*\| <\delta$ then $\|P_N\Omega(x)\| \leq \Delta\|x-x^*\|$. Assume further that $\omega(x) = \mathcal{O}(\|x-x^*\|^2)$ as $x\to x^*$. Then there exists $\bar{\varepsilon} = \bar{\varepsilon}(v^*) > 0$ and $\bar{\delta} = \bar{\delta}(v^*) > 0$ such that if $\|x-x^*\| < \bar{\varepsilon}$ and $\|\frac{x-x^*}{\|x-x^*\|}-v^*\|  < \bar{\delta}$, then equation \eqref{eq:pnm_step} has a unique solution $w$ and 
    \begin{align}
        \label{eq:keyrangebound}
        P_R(x+w - x^*) &= \mathcal{O}(\|x-x^*\|^2)\\ 
        \label{eq:keynullbound}
        P_N(x+w-x^*) &= \dfrac{1}{2}P_N(x-x^*)+\mathcal{O}(\|P_R(x-x^*)\|) \\ 
        \nonumber &+ \mathcal{O}(\|x-x^*\|^{-1}P_N\omega(x)\|)
        +\mathcal{O}(\|P_N\Omega(x)\|)+\mathcal{O}(\|x-x^*\|^2).    
    \end{align}
\end{lemma}
The authors go on to prove \cite[Theorem 1]{IzKuSo18-1}  which says that, given additional assumptions on $P_N\Omega$, $P_N\omega$, and $v^*\in N$ , that for sufficiently small $\varepsilon>0$ and $\delta>0$, the region 
\begin{align}
    \label{eq:starlike-domain}
    \hat{W} = \left\{x : \|x-x^*\| < \epsilon, \hspace{0.25em} \left\| \dfrac{x-x^*}{\|x-x^*\|} - v^*\right\| < \delta \right\}
\end{align}
is a domain of convergence for pNM. That is, if $\|x_0-x^*\| < \varepsilon$ and $\|(x_0-x^*)/\|x_0-x^*\| - v^*\| < \delta$, for sufficiently small $\varepsilon>0$ and $\delta>0$, then $\{x_k\}$ generated by pNM remains in $\hat{W}$ and converges to $x^*$ linearly. It is precisely the structure of the error expansions above shared by general pNM and standard Newton that enables the analysis in \cite{DaPo23} to be applied in this work. 

\section{Anderson Acceleration and Safeguarding}
\label{sec:aa_sg}
In this section, we recall the Anderson acceleration algorithm and the safeguarding scheme of interest in this work. 
\subsection{Anderson Acceleration}

Given a generic fixed point iteration $x_{k+1}^{fp} = g(x_k)$, we define the residual at step $k+1$ as $w_{k+1}:= g(x_k)-x_k$. Anderson acceleration of depth $m \geq 1$ generates a new iterate $x_{k+1}^{AA}$ for each $k\geq 1$ according to the following algorithm. \\
{ 
\begin{minipage}{1.0\linewidth}
 \begin{algorithm}[H]
\begin{algorithmic}[1]
\caption{Anderson Acceleration with depth $m$ \big(AA(m)\big)}
\label{alg:anderson}
\State{Choose $x_0\in\mathbb{R}^n$ and $m\geq 0$.
Set $w_1=g(x_0)-x_0$, and $x_1=x_0+w_1$.}
\For {$k=1,2,...$}
\State $m_k\gets \text{min}\{k,m\}$
\State $w_{k+1}\gets g(x_k)-x_k$
 \State $F_k = \big( (w_{k+1}-w_k)\cdots (w_{k-m+2}-w_{k-m+1})\big)$
 \State $E_k = \big( (x_k-x_{k-1})\cdots (x_{k-m+1}-x_{k-m})\big)$
 \State $\gamma_{k+1}\gets \text{argmin}_{\gamma\in\mathbb{R}^m} \|w_{k     +1}-F_k\gamma\|_2^2$
 \State $x_{k+1}^{AA}\gets x_k+w_{k+1}-(E_k+F_k)\gamma_{k+1}$
 \EndFor
 \end{algorithmic}
 \end{algorithm}
 \end{minipage}
}
\\
The presentation of Anderson acceleration in Algorithm \ref{alg:anderson} is one of several equivalent forms found in the literature. Alternate forms can be found in the book \cite{aa_for_pdes}, which also contains many of the key results from the Anderson literature. Most relevant to this article is the quantity known as the  
{\it optimization gain} $\theta_{k+1}$, which, following \cite{EPRX19} and \cite{PoRe21},  we define as
\begin{align}
    \theta_{k+1} := \dfrac{\|w_{k+1}-F_k\gamma^{k+1}\|}{\|w_{k+1}\|}.
\end{align}
The current theory of Anderson acceleration says that the smaller the optimization gain, the greater the acceleration. In \cite{DaPo23}, and Theorem \ref{thm:main} in Section \ref{subsec:lin-cvg}, it is shown that the optimization gain also determines acceleration when applied to singular problems. 

\subsection{Adaptive $\gamma$-Safeguarding}

The adaptive $\gamma$-safeguarding algorithm, defined in Algorithm \eqref{alg:adapgsg}, was first developed in \cite{DaPoRe24}. It is a modified version of  \cite[Algorithm 2]{DaPo23} that is designed to automatically scale an Anderson step towards the non-accelerated step when appropriate. The idea is that if the underlying fixed-point iteration is converging quickly, which is determined in practice by monitoring $\|w_{k+1}\|/\|w_k\|$, we want to scale the Anderson step towards the non-accelerated step $x_k+w_{k+1}$. If the problem is singular, then scaling the iterate this way ensures that we remain within the domain of convergence. If the problem is nonsingular, then this scaling reduces the effect Anderson has on the order of convergence \cite{ReXi23}. This leads to a local convergence result in the singular case, and has the additional feature that if the problem is nonsingular, adaptive $\gamma$-safeguarding detects this, and turns off Anderson acceleration which results in local quadratic convergence when the underlying fixed point iteration is Newton's method \cite[Corollary 3.4]{DaPoRe24}. Applying Algorithm \ref{alg:adapgsg} to Algorithm \ref{alg:pnm} with algorithmic depth $m=1$ results in Algorithm \ref{alg:gsgpnm}. Local convergence of Algorithm \ref{alg:gsgpnm} is analyzed in the next section. 

If the depth $m > 1$, one may still apply Algorithm \ref{alg:adapgsg} {\it asymptotically}. To do this, use Algorithm \ref{alg:anderson} with Algorithm \ref{alg:pnm} until $\|w_{k+1}\| < \tau$, for some user-chosen threshold $\tau > 0$. Then set $m=1$ and run Algorithm \ref{alg:gsgpnm} until convergence. This technique is demonstrated in Section \ref{sec:numerics}. Note that the convergence theory in Section \ref{sec:cvg_analysis} applies provided that Algorithm \ref{alg:gsgpnm} is active when the iterates reach the domain of convergence.
We also note that Algorithm \ref{alg:adapgsg} may accept any real number $r$, but the convergence results developed in Section \ref{sec:cvg_analysis} require that $r\in(0,1)$. There was no meaningful advantage observed when testing Algorithm \ref{alg:gsgpnm} with values of $r>1$; and, when $r\leq0$, $r_{k+1} = r$ at each step. Figures \ref{fig:mu98} and \ref{fig:mu94} in Section \ref{subsec:coanda} show results for Algorithm \ref{alg:adapgsg} applied to Newton's method asymptotically with $r=0$. Setting $r=0$ means that $x_{k+1}^{fp}=x_{k+1}^{AA}$. This choice of $r$ is advantageous if the iterates are close to the solution and the problem is known to be nonsingular, in which case we do not want to interfere with Newton's fast convergence. The power of adaptive $\gamma$-safeguarding is that, even if we do not know that the iterates are close to the solution, or that the problem is nonsingular, adaptive $\gamma$-safeguarding can detect this for us.\\
\begin{minipage}{1.0\linewidth}
\begin{algorithm}[H] \caption{$\lambda$ $ =$ Adaptive $\gamma$-safeguarding($\gamma_{k+1}$, $w_{k+1}$, $w_k$, $r$)}
\label{alg:adapgsg}
\begin{algorithmic}[1]
\State $\eta_{k+1} \gets \|w_{k+1}\|/\|w_k\|$
\State $r_{k+1} \gets \text{min}\{\eta_{k+1},r\}$
\State {$\beta_{k+1} \gets r_{k+1}\eta_{k+1}$}
\If{$\gamma_{k+1}=0$ {\bf or } $\gamma_{k+1}\geq 1$}
\State $\lambda \gets 0$
\ElsIf {$|\gamma_{k+1}|/|1-\gamma_{k+1}|>\beta_{k+1}$} 
\State $\lambda \gets \dfrac{\beta_{k+1}}{\gamma_{k+1}\left(\beta_{k +1}+\text{sign}(\gamma_{k+1})\right)}$
\Else {\hspace{0.25em} $\lambda \gets 1$} 
\EndIf
\end{algorithmic}
\end{algorithm}
\end{minipage}
\vspace{0.5em}
\begin{minipage}{1.0\linewidth}
\begin{algorithm}[H]
  \begin{algorithmic}[1]
  \caption{Adaptive $\gamma$-Safeguarded pNM ($\gamma$AApNM$(m=1,r)$)}
  \label{alg:gsgpnm}
  \State{Choose $x_0\in\mathbb{R}^n$ and $r\in(0,1)$. Set $w_1=-(f'(x_0)+\Omega_0)^{-1}(f(x_0)-\omega_0)$\\ and $x_1=x_0+w_1$}
  \For{$k=1,2,...$}
  \State $w_{k+1}\gets -(f'(x_k)+\Omega_k)^{-1}(f(x_k)-\omega_k)$
  \State $\gamma_{k+1}\gets (w_{k+1}-w_k)^Tw_{k+1}/\|w_{k+1}-w_k\|_2^2$
  \State $\lambda \gets$ Adaptive $\gamma$-safeguarding($\gamma_{k+1}$, $w_{k+1}$, $w_k$, $r$)
    \State $x_{k+1} \gets x_k + w_{k+1} - \lambda\gamma_{k+1}(x_k-x_{k-1}+w_{k+1}-w_k)$
  \EndFor
  \end{algorithmic}
\end{algorithm}
\end{minipage}

\section{Convergence Analysis}
\label{sec:cvg_analysis}

The goal of this section is to  analyze Algorithm \ref{alg:gsgpnm} in the case where the underlying pNM converges linearly (Section \ref{subsec:lin-cvg}) and when it converges superlinear (Section \ref{subsec:suplin-cvg}). The key assumption in Section \ref{subsec:lin-cvg} is 2-regularity. As stated in Section \ref{sec:pnm}, if $f$ is 2-regular, the structure of the error expansions of pNM iterates (see \cite{IzKuSo18-1,IzKuSo18-2}) is identical to that of standard Newton \cite{DeKeKe83,Gr80}. Convergence therefore follows immediately from \cite[Theorem 6.1]{DaPo23}. In Section \ref{subsec:suplin-cvg}, we show that if the underlying fixed-point iteration that we are applying AA to is locally superlinearly convergent, then AA with Algorithm \ref{alg:adapgsg} will not reduce the order of convergence. We thus extend \cite[Corollary 3.4]{DaPoRe24} for Newton's method applied to nonsingular problems, and we gain a result that says $\gamma$-safeguarded AALM converges superlinearly under the local error bound condition. 

\subsection{Linear Convergence}
\label{subsec:lin-cvg}

In this section, we extend the convergence results from \cite{DaPo23,DaPoRe24} to $\gamma$AApNM$(m,r)$ where $m \geq 1$ is the algorithmic depth of AA. In this subsection, we will assume the following.

\begin{assumption}
\label{assumptions}
\phantom{.}
\begin{enumerate}[itemsep=0.1em]
    \item The function $f:\mathbb{R}^n\to\mathbb{R}^n$ is $C^2$ and $f''(x)-f''(x^*) = \mathcal{O}(\|x-x^*\|)$ as $x$ approaches $x^*$. 
    \item $f'(x^*)$ is singular, with one dimensional null space $N = \text{span } \{v^*\}$ with $\|v^*\|=1$ and range space $R$. Moreover, we assume $\mathbb{R}^n = N \oplus R$ (see Section \ref{sec:pnm}). 
    \item $f$ is 2-regular at $x^*$ along $v^*\in N$. That is, $f'(x^*)+P_Nf''(x^*)(v^*,P_N(\cdot))$ is nonsingular as a map from $N$ to $N$.  
    \item The perturbation terms, $\Omega_k$ and $\omega_k$, satisfy the estimates in Lemma \ref{lem:izmailov1}, as well as $P_N\Omega_k = \mathcal{O}(\|P_R(x_k-x^*)\|)+\mathcal{O}
    (\|x_k-x^*\|^2)$ as $x_k-x^*$ approaches 0, and $P_N\omega_k = \mathcal{O}(\|x_k-x^*\|\,\|P_R(x_k-x^*)\|)+\mathcal{O}(\|x_k-x^*\|^3)$ as $x_k-x^*$ approaches 0. 
\end{enumerate}
\end{assumption}

Assumptions 1-4 imply the existence of a starlike domain $\hat{W}$ in which pNM converges to $x^*$ from any $x_0\in \hat{W}$, and that, if $x_k\in \hat{W}$, then the error $x_k+w_{k+1}-x^*$, where $w_{k+1}$ solves $(f'(x_k)+\Omega_k)d = -f(x_k)+\omega_k$, may be decomposed as 

\begin{align}
    \label{eq:pnm_exp}
    x_k+w_{k+1}-x^* = \dfrac{1}{2}P_N(x_k-x^*) + B_k + C_k,
\end{align}

where $P_RB_k=0$, $\|B_k\| = \mathcal{O}(\|P_R(x_k-x^*)\|)$, and $\|C_k\| = \mathcal{O}(\|x_k-x^*\|^2)$ \cite[Theorem 1]{IzKuSo18-1}. 
Note that, by writing $x_k+w_{k+1}-x^* = P_Ne_k + P_Re_k + w_{k+1}$, it follows that 

\begin{align}
    \label{eq:w_exp}
    w_{k+1} = -\dfrac{1}{2}P_Ne_k + B_k + C_k. 
\end{align}

Equation \eqref{eq:w_exp} provides a heuristic explanation for why Anderson acceleration is so successful for singular problems, especially when $\dim N = 1$. Near $N$ and $x^*$, $w_{k+1} \approx -(1/2)P_Ne_k$. Thus the optimization step, step 7, in Algorithm \ref{alg:anderson} can be viewed as a perturbed least squares problem minimizing the error along the null space. This can yield significant acceleration. 

We now state the extension of \cite[Theorem 6.1]{DaPo23} to general pNM. The proof in \cite{DaPo23} relied soley on safeguarding and the error expansion of the errors. Since the error expansion for pNM is the same as those of standard Newton when $f$ is 2-regular, the proof is identical, and we omit it here.

\begin{theorem}
    \label{thm:main}
    Let all assumptions hold. Then if $x_0\in \hat{W}$ and $x_1 = x_0 + w_{1}$, where $w_{1}$ solves \eqref{eq:pnm_update}, the sequence $\{x_{k+1}\}$ generated by Algorithm \ref{alg:gsgpnm}, with $m=1$ asymptotically, remains in $\hat{W}$ and converges to $x^*$. Further,     
    \begin{align}
        \|P_Re_{k+1}\| &\leq C\max\{\|e_{k}^2\|,\|e_{k-1}\|^2\}, \\ 
        \|P_Ne_{k+1}\| &\leq \kappa\theta_{k+1}\|P_Ne_k\| 
    \end{align}
    where $\kappa\in (0,1)$. 
\end{theorem}

Note that in Theorem \ref{thm:main}, it only matters that $m=1$ {\it asymptotically}. Thus, one may set $m>1$ in the preasymptotic regime, and once $\|f(x_k)\|< \tau$ or $\|w_{k+1}\| < \tau$, where $\tau$ is a user-chosen threshold, set $m=1$ and continue. As long as $\tau$ is sufficiently large so that $m=1$ once the iterates reach $\hat{W}$, then Theorem \ref{thm:main} applies. This strategy is employed in Section \ref{sec:numerics}. 

\subsection{Superlinear Convergence}
\textbf{\label{subsec:suplin-cvg}}

In this section, we are interested in analyzing what happens when the underlying pNM converges superlinearly. Theorem 3.2 and its corollary from \cite{DaPoRe24} 
combine to show that adaptive $\gamma$-safeguarding recovers local quadratic convergence when $f'(x^*)$ is nonsingular. In the perturbed Newton framework, there may be other conditions that ensure the underlying pNM is converging superlinearly, such as the local error bound condition for the LM method. Adaptive $\gamma$-safeguarding is amenable to any of these. This idea is stated rigorously in Theorem \ref{thm:generalized_nonsingular_case}, but we first state Lemma 3.2 from \cite{DaPoRe24}, which we will use to prove  Theorem \ref{thm:generalized_nonsingular_case}. As mentioned in \cite{DaPoRe24}, the proof of Lemma \ref{lem:gammabound} consists of simply walking through the various cases in Algorithm \ref{alg:adapgsg}. 

 \begin{lemma}{(\cite[Lemma 3.2]{DaPoRe24})}\label{lem:gammabound}
  Let $\eta_{k+1} = \|w_{k+1}\|/\|w_k\|$ and $r\in (0,1)$.
  Define 
  $r_{k+1} := \operatorname{min}\{\eta_{k+1},r\}$ and
  $\beta_{k+1} := r_{k+1}\eta_{k+1}$ as in
  Algorithm \ref{alg:adapgsg}.
  Let $\lambda_{k+1}$ be the value computed by
  Algorithm \ref{alg:adapgsg} at iteration $k$, and $\gamma_{k+1}$ the value computed in line 7 of Algorithm \ref{alg:anderson}.  If $\eta_{k+1} < 1$, then $\lambda_{k+1}\gamma_{k+1}$ satisfies
  \begin{align}\label{eq:gammabound1}
  |\lambda_{k+1}\gamma_{k+1}| \leq \frac{\beta_{k+1}}
  {1-\beta_{k+1}}.
  \end{align}
  when $\lambda_{k+1} = 1$, and
  \begin{align}\label{eq:gammabound2}
  |\lambda_{k+1}\gamma_{k+1}| = \frac{\beta_{k+1}}
  {1+\operatorname{sign}(\gamma_{k+1})\beta_{k+1}}.
\end{align}
when $\lambda_{k+1} < 1$.
\end{lemma}

We now state and prove Theorem \ref{thm:generalized_nonsingular_case}, which provides conditions under which adaptive $\gamma$-safeguarded AA preserves the order of convergence of the underlying fixed-point method. The assumptions are meant to describe the case where the underlying fixed-point method (note that a pNM can be viewed as a fixed-point method) is converging superlinearly within a neighborhood of a solution $x^*$. The conclusion then states that adaptive $\gamma$-safeguarding recognizes this, and tunes the Anderson step down to preserve the order of convergence. 

\begin{theorem}
\label{thm:generalized_nonsingular_case}
    Let $\{x_{k+1}\}$ be a sequence defined by $x_{k+1} = x_k + w_{k+1}$ for $k\geq 0$, where $w_{k+1} = g(x_k)-x_k$ and $g(x^*)=x^*$. Let $e_{k}^{fp} = x_k+w_{k+1}-x^*$. Suppose that there exists a real number $\rho\in(0,1)$ and constant $C>0$ such that from any $x_0\in B_{\rho}(x^*)$, $x_k$ converges to $x^*$ and, for all $k\geq 0$, we have $x_k+w_{k+1} \in B_{\rho}(x^*)$ and $\|e_{k+1}^{fp}\|\leq C\|x_k-x^*\|^{\mu}$ for some real number $\mu \in (1,2]$. Assume further that the algorithmic depth $m=1$ once $x_k\in B_{\rho}(x^*)$, and $\|w_{k+1}\|/\|w_k\|\leq r$, where $r$ is the user-chosen parameter in Algorithm \ref{alg:adapgsg}. Then if $\{x_k^{\gamma AA}\}$ denotes the sequence generated by Algorithm \ref{alg:anderson} with adaptive $\gamma$-safeguarding (Algorithm \ref{alg:adapgsg}), and $x_k,x_{k-1}\in B_{\rho}(x^*)$, it follows that $\|x_{k+1}^{\gamma AA} - x^*\| \leq \hat{C}\|x_k-x^*\|^{\mu}$.
\end{theorem}

\begin{proof}
    In this proof, we'll let $e_{k+1}^{\gamma AA} = x_{k+1}^{\gamma AA}-x^*$. Let $x_k$ and $x_{k-1}\in B_{\rho}(x^*)$. Note that $\|e_{k+1}^{\gamma AA}\| \leq \|e_k^{fp}\|+\|e_{k+1}^{\gamma AA}-e_k^{fp}\|$. Since 
    \begin{align}
        \label{eq:fp-mu-bound}
        \|e_{k+1}^{fp}\|\leq C\|x_k-x^*\|^{\mu}, 
    \end{align}it
    suffices to bound $\|e_{k+1}^{\gamma AA}-e_k^{fp}\|$ in terms of $\|x_k-x^*\|^{\mu}$. Because we assume that $m=1$ within $B_{\rho}(x^*)$, it follows that $e_{k+1}^{\gamma AA} = e_{k+1}^{fp}-\lambda_{k+1}\gamma_{k+1}(e_{k+1}^{fp}-e_{k}^{fp})$. By Lemma \ref{lem:gammabound} and the assumption that $x_{k},x_{k-1}\in B_{\rho}(x^*)$, 
    \begin{align}
        \label{eq:gamma-fp-bound}
        \|e_{k+1}^{\gamma AA}-e_{k+1}^{fp}\|\leq \dfrac{\beta_{k+1}}{1-\beta_{k+1}}\left(1+\dfrac{\|x_k-x^*\|^{\mu}}{\|x_{k-1}-x^*\|^{\mu}}\right)\|x_{k-1}-x^*\|^{\mu}.
    \end{align}
    Using $w_{k+1}=e_{k+1}^{fp}-(x_k-x^*)$ and $\|e_{k+1}^{fp}\|\leq C\|e_k\|^{\mu} \leq C\rho^{\mu}$, we have 
    \begin{align}
        \left(1-C\rho^{\mu-1}\right)\|x_k-x^*\|\leq\|w_{k+1}\|\leq \left(1+C\rho^{\mu-1}\right)\|x_k-x^*\|.
    \end{align}
    Thus, for sufficiently small $\rho$, $\|w_{k+1}\|/\|w_{k}\| <1$ implies $\|x_k-x^*\|/\|x_{k-1}-x^*\| < 1$. Moreover,
    \begin{align}
        \dfrac{\|w_{k+1}\|}{\|w_k\|}\leq \left(\dfrac{1+C\rho^{\mu-1}}{1-C\rho^{\mu-1}}\right)\dfrac{\|x_k-x^*\|}{\|x_{k-1}-x^*\|}.
    \end{align}
    Noting that $\beta_{k+1} = \|w_{k+1}\|^2/\|w_k\|^2 \leq r^2$ by assumption, and writing $C_1 = (1+C\rho^{\mu-1})(1-C\rho^{\mu-1})^{-1}$, equation \eqref{eq:gamma-fp-bound} becomes
    \begin{align}
        \label{eq:aa-fp-gap}
        \|e_{k+1}^{\gamma AA}-e_{k+1}^{fp}\|\leq \dfrac{C_1^2}{1-r^2}\left(1+\dfrac{\|x_k-x^*\|^{\mu}}{\|x_{k-1}-x^*\|^{\mu}}\right)\|x_k-x^*\|^2\|x_{k-1}-x^*\|^{\mu-2}.
    \end{align}
    Combining equations \eqref{eq:fp-mu-bound} and \eqref{eq:aa-fp-gap} yields 
    \begin{align}
        \label{eq:final-bound}
        \|e_{k+1}^{\gamma AA}\|\leq \left(C+\dfrac{C_1^2}{1-r^2}\left(1+\dfrac{\|x_k-x^*\|^{\mu}}{\|x_{k-1}-x^*\|^{\mu}}\right)\left(\dfrac{\|x_k-x^*\|}{\|x_{k-1}-x^*\|}\right)^{2-\mu}\right)\|x_k-x^*\|^{\mu}.
    \end{align}
    Since $\|x_k-x^*\|/\|x_{k-1}-x^*\|<1$, we may set the term in parentheses on the right-hand-side of equation \eqref{eq:final-bound} to $\hat{C}$. This completes the proof. 
    \end{proof}

    A few remarks are in order. First, the assumption that the algorithmic depth $m$ in Algorithm \ref{alg:anderson} is set to one once $x_k$ is close to $x^*$ motivates the practical {\it asymptotic safeguarding} strategy demonstrated in \cite{DaPoRe24} and Section \ref{sec:numerics} of this work. 
    
    Second, the assumptions of Theorem \ref{thm:generalized_nonsingular_case} together simply say that the underlying fixed point iteration $x_{k+1} = g(x_k)$ is rapidly approaching the fixed point $x^*$. Thus, regardless of the particulars of $g$, Algorithm \ref{alg:adapgsg} will detect this rapid convergence and respond by scaling the Anderson iterates $x_{k+1}$ towards $g(x_k)$ in such a way that the order of convergence is preserved. Thus, Theorem \ref{thm:generalized_nonsingular_case} applies to both Newton's method if the Jacobian is nonsingular, and to the Levenberg-Marquardt method under the local error bound condition \cite{KaYaFu04}. In the class of problems for which Levenberg-Marquardt is appropriate, it's possible that there is not an isolated solution $x^*$, but rather a {\it set} of solutions $X^*$. One then wishes to measure the convergence of a sequence $\{x_k\}$ to the set $X^*$ by analyzing $\text{dist}(x_k,X^*):=\inf_{y\in X^*}\|x_k-y\|$. The conclusion of Theorem \ref{thm:generalized_nonsingular_case} in this case says that 
    \begin{align}
    \text{dist}(x_{k+1}^{\gamma AA},X^*)\leq \hat{C}\text{dist}(x_k,X^*)^{\mu}. 
    \end{align}
    
    The third remark concerns an apparent shortcoming of the conclusion of Theorem \ref{thm:generalized_nonsingular_case}. If the fixed point iteration converges faster than quadratic, then the final bound in \eqref{eq:final-bound} breaks down since $2-\mu < 0$ and $\|x_{k-1}-x^*\|/\|x_k-x^*\|$ will be large in the asymptotic regime. However, the ``2" in $2-\mu$ comes from the choice of $\beta_{k+1}$ in Algorithm \ref{alg:adapgsg}. We could just as well choose $\beta_{k+1} = r_{k+1}(\|w_{k+1}\|/\|w_k\|)^{p-1}$, for some $p\geq 2$, and then we would have $p - \mu$ in equation \eqref{eq:final-bound}. We summarize this in the following corollary. 

    \begin{corollary}
        The conclusion of Theorem \ref{thm:generalized_nonsingular_case} holds under the same assumptions if $\beta_{k+1}=r_{k+1}\|w_{k+1}\|/\|w_k\|$ in Algorithm \ref{alg:adapgsg} is replaced with 
        \begin{align} 
        \beta_{k+1}=r_{k+1}(\|w_{k+1}\|/\|w_k\|)^{p-1}
        \end{align}
         for any real number $p\geq 2$. 
    \end{corollary}

In the next subsection we discuss how the theory developed in the previous two subsections apply to the special cases of the Levenberg-Marquardt method. 

\subsection{Special Case: Levenberg-Marquardt}

As discussed in the introduction, the Levenberg-Marquardt (LM) method is a technique for solving nonlinear least-squares problems. It can also be used as a root-finding algorithm when the objective function's minimum value is zero. Given $f:\mathbb{R}^n\to\mathbb{R}^n$ and an initial guess $x_0$, the goal at step $k\geq 0$ is to minimize $\|f(x_k)+f'(x_k)d\| + \mu_k\|d\|$ over all $d\in\mathbb{R}^n$. This leads to the following algorithm. \\
\begin{minipage}{1.0\linewidth}
\begin{algorithm}[H]
\begin{algorithmic}[1]
\caption{Levenberg-Marquardt (LM)}
\label{alg:lm}
\State{Choose $x_0\in\mathbb{R}^n$ and $\mu_0\in\mathbb{R}$.}
\For{$k=0,1,...$}
\State $w_{k+1}\gets -(f'(x_k)^Tf'(x_k)+\mu_k I)^{-1}f'(x_k)^Tf(x_k)$
\State $x_{k+1}\gets x_k+w_{k+1}$
\EndFor
\end{algorithmic}
\end{algorithm}
\end{minipage}
\vspace{0.5em}

The parameter $\mu_k$ is the LM parameter. If $X^*$ is the set of minimizers of $\|f(x)\|^2$, and $\mu_k = \mu_0\|f(x_k)\|^2$, then the sequence generated by Algorithm \ref{alg:lm} converges quadratically to the solution set $X^*$ \cite[Theorem 2.6]{KaYaFu04} under the local error bound condition. If $f'(x^*)$ is 
nonsingular, then the local 
error condition holds, and it follows that LM converges locally quadratically for nonsingular problems. 
To see this, suppose that $f$ is $C^2$, and let $x^*\in X^*$ such that $f'(x^*)$ is nonsingular. Let $m=\|f'(x^*)^{-1}\|^{-1}/2$. Then there exists a neighborhood $B_R(x^*)$ such that if $x\in B_R(x^*)$, then $f'(x)$ is invertible with $\|f'(x)^{-1}\|\leq 1/m$ and $\|f'(x)y\|\geq m\|y\|\,$ for any $y$ \cite[Chapter 4]{Ke95}. Let $M=\max\|f''(x)\|$ over $B_R(x^*)$, Taylor expand $f(x^*)=0$ about $x\in B_R(x^*)$, and apply the reverse triangle inequality to get 
\begin{align}
    \|f(x)\|\geq \left(m-\dfrac{M}{2}\|x-x^*\|\right)\|x-x^*\| \geq  \left(m-\dfrac{MR}{2}\right)\|x-x^*\|.
\end{align}
Thus $\text{dist}(x,X^*)\leq c\|f(x)\|$ where $c=(m-MR/2)^{-1}.$
%A corollary is that LM converges locally quadratically for nonsingular problems. 
In \cite{IzKuSo18-1}, it was shown that if the error bound condition is replaced with 2-regularity, then LM converges locally linearly in a starlike domain about $x^*$ just like standard Newton. It is also shown in \cite{IzKuSo18-1} that, if $f$ is 2-regular along some $v\in N$, then the local error bound condition does not hold. In the case where the local error bound condition holds or the problem is nonsingular, Theorem \ref{thm:generalized_nonsingular_case} says that Anderson acceleration with $\gamma$-safeguarding will not reduce the order of convergence. 

To discuss the 2-regular case, first note that if $f'(x_k)^T$ is invertible, then LM can be viewed as a pNM where $\Omega_k = \mu_kf'(x_k)^{-T}$ and $\omega_k=0$ \cite{IzKuSo18-1}. Taking $\omega_k\neq 0$ yields an inexact LM method, but here we will assume $\omega_k=0$. Lemma \ref{lem:izmailov1} and Theorem \ref{thm:main} together require that $\Omega(x_k) = \mathcal{O}(\|x_k-x^*\|)$, $\|P_N\Omega(x_k)\|\leq \Delta\|x_k-x^*\|$, where $\Delta>0$ is given, and $P_N\Omega(x_k) = \mathcal{O}(\|P_R(x_k-x^*)\|) +\mathcal{O}(\|x_k-x^*\|^2).$ Under Assumption \ref{assumptions}, we have $P_Nf'(x_k)^{-T} = \mathcal{O}(\|x_k-x^*\|^{-1})$ within the starlike domain of convergence \cite{DeKeKe83,Gr80}. Thus, by choosing $\mu_k = \mu_0\|f(x_k)\|^2$ as in \cite{KaYaFu04}, we ensure that each condition holds for convergence in the 2-regular case. The authors of \cite{IzKuSo18-1} note that one may take $\mu_k = \mu_0\|f(x_k)\|^{p}$ for any $p\geq 2$. Unless stated otherwise, we choose $\mu_k$ according to \cite{KaYaFu04} in Section \ref{sec:numerics}. 

These theoretical results suggest that Anderson acceleration can be a useful tool to pair with the LM method both when the local error bound condition holds and when it does not. This does not seem to be common practice given the lack of articles in the literature applying Anderson to LM. 

\section{Numerical Experiments}
\label{sec:numerics}

In this section, we demonstrate safeguarded Anderson acceleration applied to two particular pNMs: Newton  and LM (Algorithm \ref{alg:lm}). We also demonstrate the inexact Newton and inexact LM methods in subsection \ref{subsec:chand}. All experiments were performed on an M1 MacBook Pro with 8 GB of RAM. The code used in Subsections \ref{subsec:chand} and \ref{subsec:aalm} is available on the \href{https://github.com/mdallas1/AApNM}{author's Github page}. We will let $\gamma$AANewt$(m,r)$ and $\gamma$AALM$(m,r)$ denote, respectively, Anderson accelerated Newton and LM with adaptive $\gamma$-safeguarding, algorithmic depth $m$, and safeguarding parameter $r$. The inexact versions will follow the same format, but with inNewt or inLM. 
There are two ways one can implement $\gamma$-safeguarding. When the algorithmic depth $m=1$ in Algorithm \ref{alg:anderson}, $\gamma$-safeguarding may be activated for the entire solve, or it can be activated once the error estimator falls below a user-chosen threshold $\tau$. The former strategy is called {\it preasymptotic safeguarding} in \cite{DaPoRe24}, and the latter is {\it asymptotic safeguarding}. Asymptotic safeguarding is supported by the theory, but preasymptotic safeguarding has been observed to outperform standard AANewt for certain problems with the right parameter choice \cite{DaPoRe24}. If the algorithmic depth $m >1$, then one must use asymptotic safeguarding. Aside from some experiments with the Chandrasekhar H-equations, we will use asymptotic safeguarding in the demonstrations with $\tau = 10^{-1}$. It is the strategy best supported by theory, and numerical experiments suggest that, generally, we want the full, unsafeguarded Anderson acceleration in the preasymptotic regime. 

\subsection{The Chandrasekhar H-Equation}
\label{subsec:chand}

The Chandrasekhar H-equation \cite{chandra60,Ke95} is an integral equation from radiative transfer theory with a parameter typically denoted by $\omega$, but we will write as $\bar{\omega}$ to distinguish it from the perturbation term $\omega_k$ in equation \eqref{eq:pnm_update}. To solve the Chandrasekhar $H$-equation is to find a continuous function $H(x):[0,1]\to\mathbb{R}$ such that 
\begin{align} 
    \label{eq:ch}
    H(x) - \left(1-\dfrac{\bar{\omega}}{2}\int_0^1 \dfrac{xH(s)}{x+s}\,ds\right)^{-1} = 0.
\end{align}
Equation \ref{eq:ch} has become a standard benchmark problem in the singular Newton literature and related areas \cite{DeKeKe83,KaYaFu04,Ke18,posc20,ToKe15}. It can be viewed as an equation on $ C[0,1]$ with solution $H\in C[0,1]$, and it can be shown that when $\bar{\omega}=1$, equation \ref{eq:ch} is singular with a one-dimensional null space \cite{Ke18}. When $\bar{\omega} \in (0,1)$, the problem is nonsingular. It is therefore a nontrivial problem that is well-understood theoretically, and thus serves as an excellent benchmark on which to test new methods for singular problems. To solve equation \eqref{eq:ch}, we follow \cite{Ke18} and discretize the integral with the composite midpoint rule using $N=10^3$ nodes; $H(x)$ is approximated by a vector of the same length. We thus obtain the $10^3\times 10^3$ system of nonlinear equations 
\begin{align}
\label{eq:discretized-ch}
    f(x) = x_j - \left(1-\dfrac{\bar{\omega}}{2N}\sum_{i=1}^{N} \dfrac{t_jx_i}{t_j+t_i}\right)^{-1}=0
\end{align}
where $t_j = (2j-1)/2N$ for $j=1,...,N$ and $x_j = H(t_j)$. For $x_0$, we selected 50 random vectors and reported the average number of iterations to convergence along with the average terminal residual. Iterations were terminated once $\|f(x)\|<10^{-8}$ or 100 iterations were reached, in which case failure was declared. For the inexact solves, we chose the forcing term according to the second choice in \cite{EisWa96}. With this choice, it is not clear that the conditions in Lemma \ref{lem:izmailov1} and Assumption \ref{assumptions} on the perturbation term $\omega_k$ are satisfied. However, the numerical results presented in Tables \ref{tab:aanewt_omega_1}-\ref{tab:aanewt_omega_08} show no significant difference between the exact solves and the inexact solves.  

Results for $\bar{\omega}=1$ are reported in Tables \eqref{tab:aanewt_omega_1} and \eqref{tab:aalm_omega_2}. The results are consistent with the theory. AA Newton outperforms both exact and inexact Newton. With Algorithm \ref{alg:adapgsg}, we still converge faster than Newton, but slightly slower than Newton without Algorithm \ref{alg:adapgsg}. This is the cost of scaling Anderson iterates towards Newton iterates in order to guarantee local convergence for singular problems. Note that with asymptotic safeguarding, the inexact solves are much less sensitive to the algorithmic depth. With AAinNewt(50), the solve failed for three initial guesses, but $\gamma$AAinNewt(50) converged. The results for LM also follow the theory. Note that AAinLM(50) failed to converge from each random initial iterate. With safeguarding, there were no failures. 

If we set $\bar{\omega}=0.8$, then equation \eqref{eq:ch} and its discretization are nonsingular. In this regime, Newton's method and standard Levenberg-Marquardt converge rapidly, and Anderson acceleration slows this convergence, especially when the algorithmic depth $m$ is large. Results are reported in Table \eqref{tab:aanewt_omega_08}. As expected from the theory, $\gamma$AALM and $\gamma$AAinLM show virtually no sensitivity to the choice of $m$, and they both outperform their unsafeguarded counterparts. We also tested the inexact solves with $\bar{\omega}=0.8$, but the results were similar to the $\bar{\omega}=1$ case, so we omit them.

\begin{table}[H]
\small
    \centering
    \begin{tabular}{|c|c|c||c|c|c|}
        \hline 
         Algorithm & Iterations & $\|f(x_k)\|$ & Algorithm & Iterations & $\|f(x_k)\|$\\
        \hline 
         Newt &16 &4.45e-09 &inNewt &16 &9.21e-09 \\
         \hline
         \multirow{2}{0.5em}{}AANewt(1) &6 &2.24e-11 &AAinNewt(1) &8 &1.14e-09 \\ $\gamma$AANewt(1,0.9) &12 &2.01e-09 &$\gamma$AAinNewt(1,0.9) &13 &3.12e-09 \\
         \hline
         \multirow{2}{0.5em}{}AANewt(5) &7 &8.04-09 &AAinNewt(5) &14 &2.47e-09 \\ $\gamma$AANewt(5,0.9) &12 &2.01-e09 &$\gamma$AAinNewt(5,0.9) &13 &3.17e-09 \\
         \hline 
         \multirow{2}{0.5em}{}AANewt(10) &7 &8.01e-09 &AAinNewt(10) &20 &3.40e-09 \\$\gamma$AANewt(10,0.9) &12 &2.01e-09 &$\gamma$AAinNewt(10,0.9) &12 &3.17e-09 \\ 
         \hline
         \multirow{2}{0.5em}{}AANewt(50) &7 &8.01e-09 &AAinNewt(50) &61 &4.35e-09 \\ $\gamma$AANewt(50,0.9) &12 &2.01e-09 &$\gamma$AAinNewt(50,0.9) &13 &3.17e-09\\
         \hline
    \end{tabular}
    \caption{Results for Anderson accelerated Newton and inexact Newton applied to the Chandrasekhar H-equation for $\bar{\omega}=1$.}%results from 2024-07-03
    \label{tab:aanewt_omega_1}
\end{table}

\begin{table}[H]
\small
    \centering
    \begin{tabular}{|c|c|c||c|c|c|}
        \hline 
         Algorithm & Iterations & $\|f(x_k)\|$ &Algorithm & Iterations & $\|f(x_k)\|$\\
        \hline 
         LM &16 &5.58e-09 &inLM &17 &4.28e-09 \\
         \hline
         \multirow{2}{0.5em}{}AALM(1) &6 &1.72e-09 &AAinLM(1) &9 &4.93e-09 \\ $\gamma$AALM(1,0.9) &12 &2.20e-09 &$\gamma$AAinLM(1,0.9) &12 &8.46e-09\\
         \hline
         \multirow{2}{0.5em}{}AALM(5) &10 &4.81e-09 &AAinLM(5) &24 &2.63e-09\\ $\gamma$AALM(5,0.9) &12 &2.20e-09 &$\gamma$AAinLM(5,0.9) &12 &8.5e-09\\
         \hline 
         \multirow{2}{0.5em}{}AALM(10) &13 &3.70e-09 &AinLM(10) &51 &3.5e-09\\ $\gamma$AALM(10,0.9) &12 &2.20e-09 &$\gamma$AAinLM(10,0.9) &12 &8.47e-09 \\
         \hline
         \multirow{2}{0.5em}{}AALM(50) &45 &2.81e-09 &AAinLM(50) &F &- \\ $\gamma$AALM(50,0.9) &12 &2.89e-09 &$\gamma$AAinLM(50,0.9) & 12 &8.47e-09\\
         \hline
    \end{tabular}
    \caption{Results for Anderson accelerated Levenberg-Marquardt and inexact Levenberg-Marquardt applied to the Chandrasekhar H-equation for $\bar{\omega}=1$.} %results from 2024-07-03
    \label{tab:aalm_omega_2}
\end{table}

\begin{table}[H]
\small
    \centering
    \begin{tabular}{|c|c|c||c|c|c|}
        \hline 
         Algorithm & Iterations & $\|f(x_k)\|$ & Algorithm & Iterations & $\|f(x_k)\|$\\
        \hline 
         Newt &4 &3.8e-14 &LM &4 &4.18e-14 \\
         \hline
         \multirow{2}{0.5em}{}AANewt(1) &5 &1.34e-11 &AALM(1) &5 &2.12e-11 \\ $\gamma$AANewt(1,0.9) &4 &2.63e-14 &$\gamma$AALM(1,0.9) &4 &2.60e-14 \\
         \hline
         \multirow{2}{0.5em}{}AANewt(5) &8 &2.0e-11 &AALM(5) &8 &8.5e-14\\ $\gamma$AANewt(5,0.9) &4 &2.62e-14 &$\gamma$AALM(5,0.9) &4 &2.60e-14 \\
         \hline 
         \multirow{2}{0.5em}{}AANewt(10) &13 &8.83e-13 &AALM(10) &13 &4.55e-14\\$\gamma$AANewt(10,0.9) &4 &2.62e-14 &$\gamma$AALM(10,0.9) &4 &2.60e-14 \\ 
         \hline
         \multirow{2}{0.5em}{}AANewt(50) &26 &5.92e-09 &AALM(50) &26 &2.03e-09 \\ $\gamma$AANewt(50,0.9) &4 &2.62e-14 &$\gamma$AALM(50,0.9) &4 &2.60e-14\\
         \hline
    \end{tabular}
    \caption{Results for Anderson accelerated Newton and Anderson accelerated Levenberg-Marquardt applied to the Chandrasekhar H-equation for $\bar{\omega}=0.8$.}%results from 2024-07-03
    \label{tab:aanewt_omega_08}
\end{table}

\subsection{Incompressible Channel Flow}
\label{subsec:coanda}
In this section we apply Algorithm \ref{alg:gsgpnm}, with $\Omega_k=\omega_k=0$, to a parameter-dependent incompressible flow problem near a bifurcation point \cite{pichithesis}. This is of interest since a bifurcation point is necessarily a singular point, and, even if the problem is nonsingular, a nearby singular problem can adversely affect the solve. Anderson acceleration can mitigate this affect. However, if the underlying method is a method that, like Newton, will converge superlinearly in a neighborhood of our solution provided the problem is nonsingular, we would like to turn off Anderson acceleration in the preasymptotic regime to not hinder convergence \cite{ReXi23}. Since adaptive $\gamma$-safeguarding, Algorithm \ref{alg:adapgsg}, is designed to do precisely this, solving parameter dependent problems near bifurcation points is an excellent test-case for the algorithms of interest in this paper. 
The domain is shown in Figure \ref{fig:coanda_boundaries}. The governing equations \eqref{coanda-model} are the Navier-Stokes equations coupled with no-slip and stress-free boundary conditions on the indicated boundaries. On $\Gamma_{\text{in}}$, we define $\mathbf{u}_{\text{in}} = 20(5-x_2)(x_2-2.5)$. We discretize using the $P_2-P_1$ Taylor-Hood finite element method with 14,406 degrees of freedom.  

\begin{equation}\label{coanda-model}
\left\{
\begin{alignedat}{2}
  -\mu \Delta \mathbf{u} + \mathbf{u} \cdot \nabla \mathbf{u} + \nabla p &= \mathbf{0}, \quad &&\text{in } \Omega,\\
  \nabla \cdot \mathbf{u} &= 0, \quad &&\text{in }\Omega,\\
  \mathbf{u} &= \mathbf{u}_{\text{in}}, \quad &&\text{on } \Gamma_{\text{in}},\\
  \mathbf{u} &= 0, \quad &&\text{on } \Gamma_{\text{wall}},\\
  -p\mathbf{n} + \mu (\nabla \mathbf{u}) \mathbf{n} &= 0, \quad &&\text{on } \Gamma_{\text{out}}.
\end{alignedat}
\right.
\end{equation}

\begin{figure}[H]
    \centering
    \includegraphics[width=0.7\linewidth]{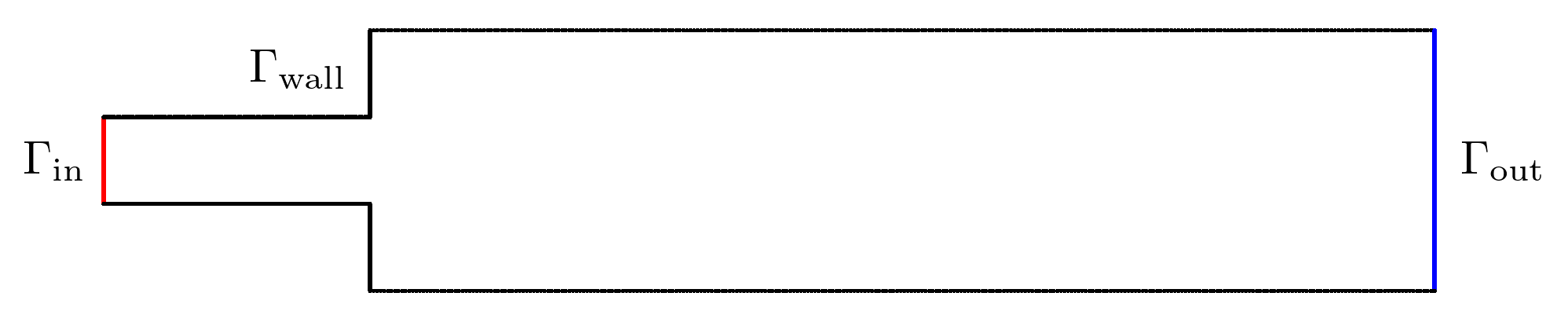}
    \caption{Channel flow domain and boundaries. The red vertical segment is $\Gamma_{\text{in}}$, the blue is $\Gamma_{\text{out}}$, and the remaining boundaries, in black, are $\Gamma_{\text{wall}}$.}
    \label{fig:coanda_boundaries}
\end{figure}

The parameter of interest is the kinematic viscosity $\mu = \text{Re}^{-1}$.  When $\mu > 1$, the stable solution is symmetric about the line $y=3.75$. Between $\mu=1$ and $\mu=0.9$, the stable solution transitions to one of two asymmetric solutions seen in Figure \ref{fig:coanda_solns}. From the standard initial guess $\mathbf{u}_0=\mathbf{0}$, Newton's method converges slowly, if at all, to the solution when $\mu$ is near the bifurcation point. Continuation can solve the convergence issues, but coupling Newton with Anderson acceleration and adaptive $\gamma$-safeguarding can improve convergence near the bifurcation without continuation. 

\begin{figure}[H]
    \centering
    \includegraphics[width=0.7\linewidth]{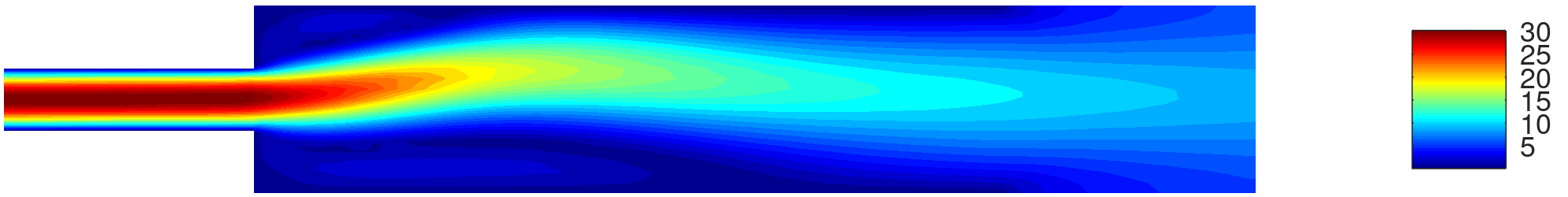}
    \includegraphics[width=0.7\linewidth]{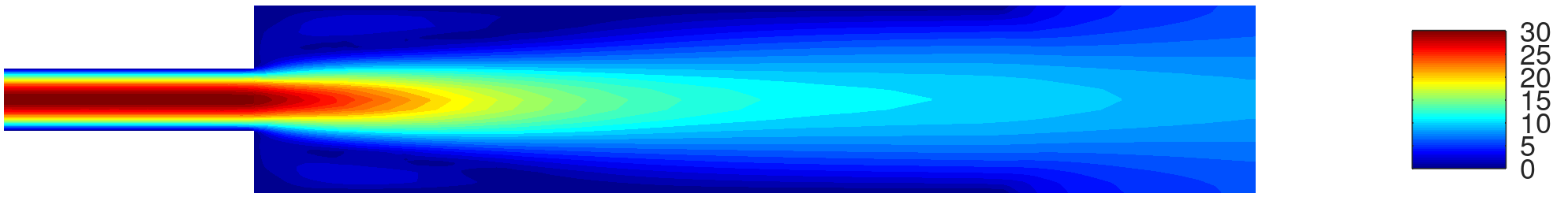}
    \includegraphics[width=0.7\linewidth]{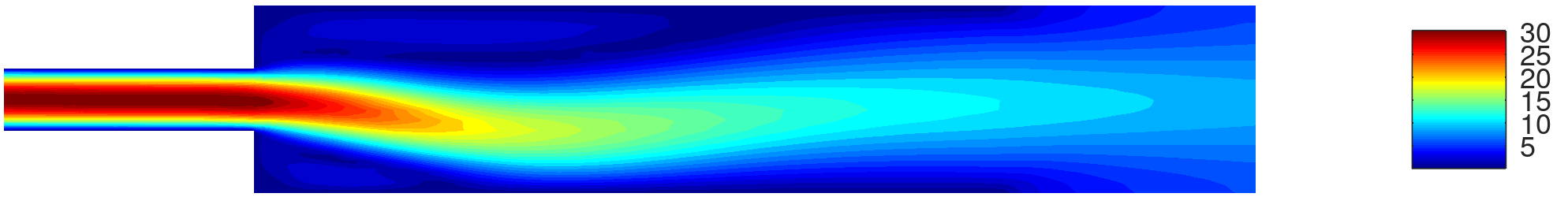}
    \caption{Stable velocity solutions for different values of $\mu$. The top and bottom figures are stable solutions when $\mu=0.9$. When $\mu=1$, the middle solution is stable.}
    \label{fig:coanda_solns}
\end{figure}

Representative numerical results are presented in Figure \ref{fig:mu98} and Figure \ref{fig:mu94}. When $\mu=0.98$, seen in Figure \ref{fig:mu98}, we are well away from the bifurcation point, and Newton's method therefore converges rapidly. Anderson acceleration of depth $m$, denoted by AANewt$(m)$ in the plots, slows convergence as $m$ is increased. With asymptotic safeguarding, this affect is mitigated even for $m=5$ and $m=10$. This is consistent with the theory. Asymptotically, we set $m=1$ and activate Algorithm \ref{alg:adapgsg} which detects that Newton's method is converges rapidly and scales the iterates towards standard Newton. We set $r=0$ in Algorithm \ref{alg:adapgsg} to ensure strong scaling towards the Newton step, but it is worth noting that, with asymptotic safeguarding, performance is insensitive to the particular choice of $r$. When $\mu=0.94$, seen in Figure \ref{fig:mu94}, we are close to the bifurcation point, and Newton takes many iterations to converge. Applying Anderson acceleration leads to much faster convergence, and asymptotic safeguarding reduces the affect of increasing the algorithmic depth. The takeaway is that when solving a problem that may be nearly singular, but this is not known a priori, using asymptotic safeguarding along with Anderson acceleration can help the solver reach the asymptotic regime, and then safeguarding ensures that the local order of convergence will not be decreased. 

\begin{figure}[H]
    \centering
    \includegraphics[width=0.49\linewidth]{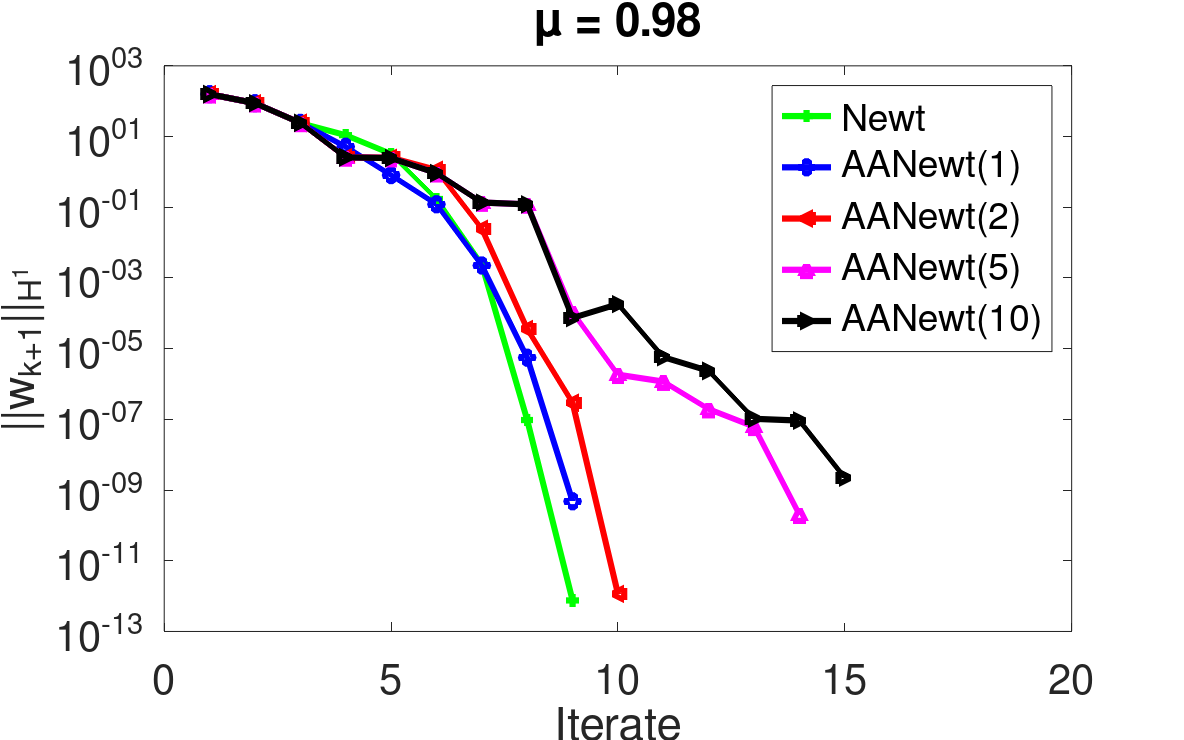}
    \includegraphics[width=0.49\linewidth]{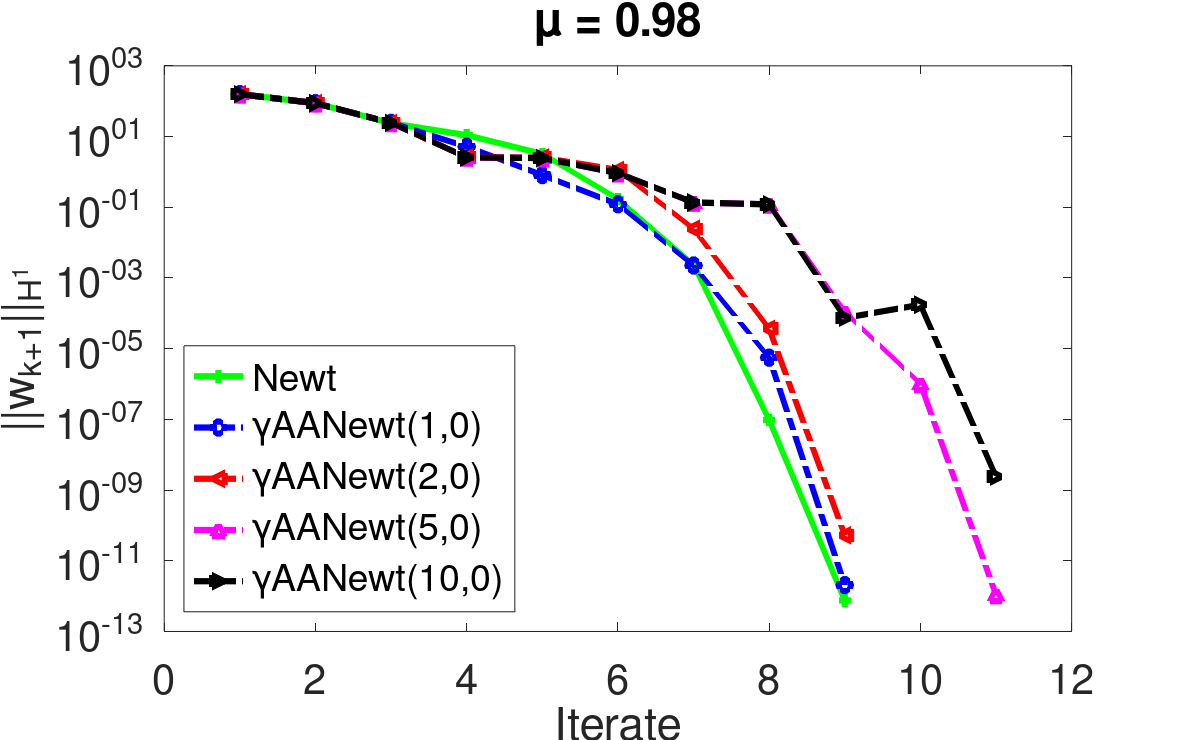}
    \caption{Residual history for channel flow problem with $\mu = 0.98$. Left: results for Newton and Newton--Anderson with algorithmic depths $m=1$, 2, 5, and 10. Right: results for $\gamma$-safeguarded Newton--Anderson applied asymptotically plotted with dash-dotted lines. We set $r=0$ in Algorithm \ref{alg:adapgsg}.} 
    \label{fig:mu98}
\end{figure}

\begin{figure}[H]
    \centering
    \includegraphics[width=0.49\linewidth]{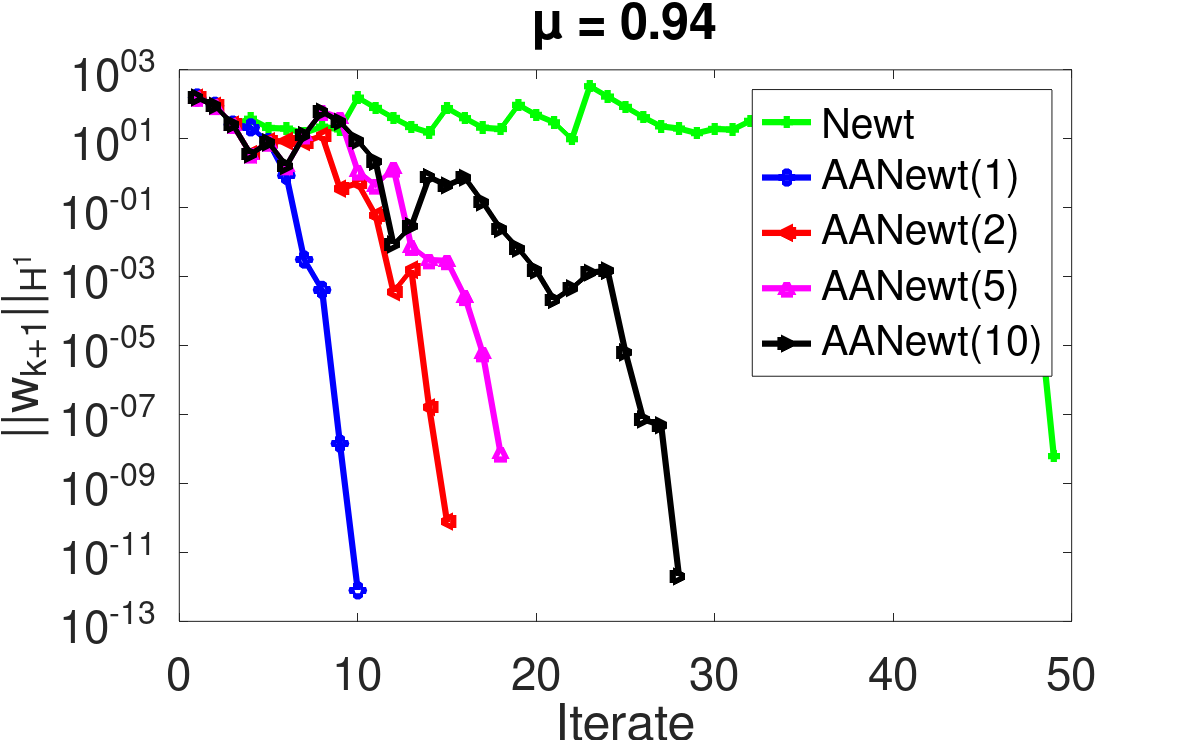}
    \includegraphics[width=0.49\linewidth]{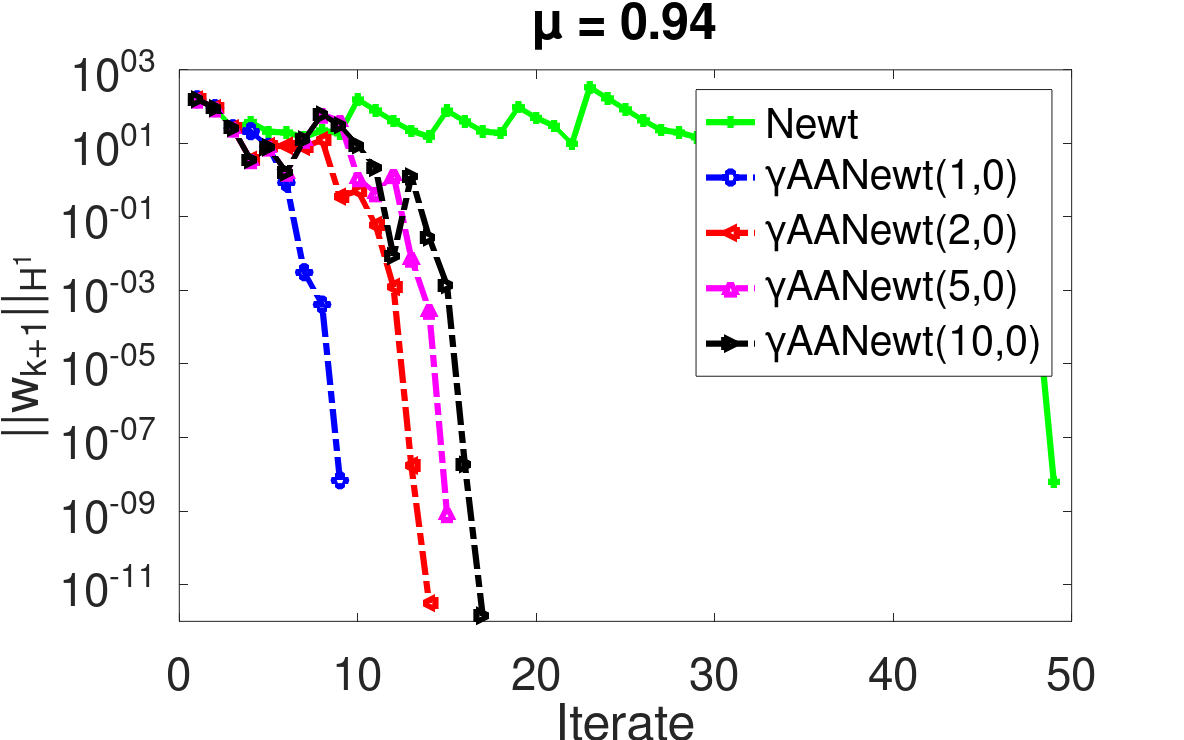}
    \caption{Residual history for channel flow problem with $\mu = 0.94$. Left: results for Newton and Newton--Anderson with algorithmic depths $m=1$, 2, 5, and 10. Right: results for $\gamma$-safeguarded Newton--Anderson applied asymptotically plotted with dash-dotted lines. We set $r=0$ in Algorithm \ref{alg:adapgsg}.}
    \label{fig:mu94}
\end{figure}

\subsection{AALM Under Local Error Bound}
\label{subsec:aalm}

The Chandrasekhar H-equation was used to demonstrate that AA and adaptive $\gamma$-safeguarded AA can accelerate convergence of LM when solving for roots of a function. These numerical results are supported by the theory since the Chandrasekhar function is 2-regular. In this section, we demonstrate that AA is also effective when the local error bound holds rather than 2-regularity. The test problems, listed below, are also interesting because the objective functions' minimum values are greater than zero. The numerical results presented here thus demonstrate that AALM and adaptive $\gamma$-safeguarded AALM can outperform LM even when LM is not used as a root-finding algorithm. The test problems correspond to examples 5.1-5.4 in \cite{BeGoSa19}. We state them below along with the initial iterate and the choice of LM parameter $\mu_k$.  
\begin{enumerate}
    \item (Beh1) $f(x) = (x_1^2+x_2^2-1,x_1^2+x_2^2-9)$, $x_0 = (0,\sqrt{5}+0.03)^T$, and $\mu_k = \|f'(x_k)^Tf(x_k)\|.$
    \item (Beh2) $f(x) = (x_1^3-x_1x_2+1,x_1^3+x_1x_2+1)^T$, $x_0 = (0.008,2.0)^T$, $\mu_k = \|f'(x_k)^Tf(x_k)\|$.
    \item (Beh3) $f(x)=((1/9)\cos(x_1)-x_2\sin(x_1),(1/9)\sin(x_1)+x_2\cos(x_1))^T$, $x_0 = (\pi,0.001)^T$, $\mu_k = 0.2$.
    \item (Beh4) $f(x) = (x_2-x_1^2-1,x_2+x_1^2+1)^T$, $x_0 = (0.01,0)^T$, $\mu_k = 5$. 
\end{enumerate}
Each solve was terminated when $\|f'(x_k)^Tf(x_k)\| < 10^{-8}$ or 100 iterations were reached, in which case failure was declared. The local error bound condition is satisfied for each test problem \cite{BeGoSa19}. We remark that solutions for Beh1, Beh2, and Beh3 are nonisolated, and it was observed during numerical experiments that AALM converged to a different solution than LM. The result for Beh1 and Beh2 are presented in Figure \ref{fig:beh1-and-2}. In these problems, LM converges rapidly, and AALM with Algorithm \ref{alg:adapgsg} recognize this and respond accordingly. Note that the safeguarded versions of AALM outperform the unsafeguarded counterparts. In Figure \ref{fig:beh3}, we see that for Beh3 LM converges linearly; and, although the safeguarded versions of AALM converge, they barely converge faster than LM. AALM without safeguarding is more affective here. Similar results from Beh4 are shown in Figure \ref{fig:beh4}. In fact, when $\mu_k = \|f'(x_k)^Tf(x_k)\|$, LM fails, but we can recover convergence using AALM {\it without safeguarding}. Thus, in practice, one may need full unsafeguarded Anderson acceleration to recover or accelerate convergence, but if the problem is known to be locally nonlinear, but near a singular problem, one should use Anderson acceleration with Algorithm \ref{alg:adapgsg} to ensure the most rapid local convergence. 

\begin{figure}[H]
    \centering
\includegraphics[width=0.49\linewidth]{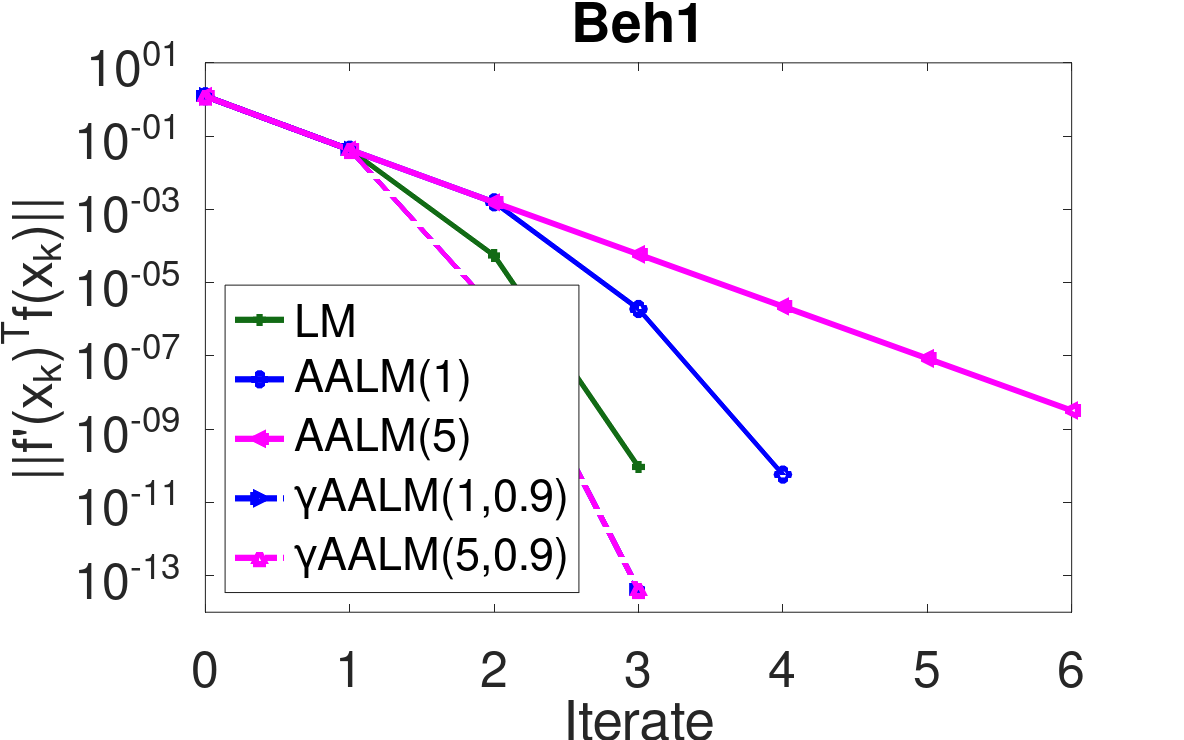}
\includegraphics[width=0.49\linewidth]{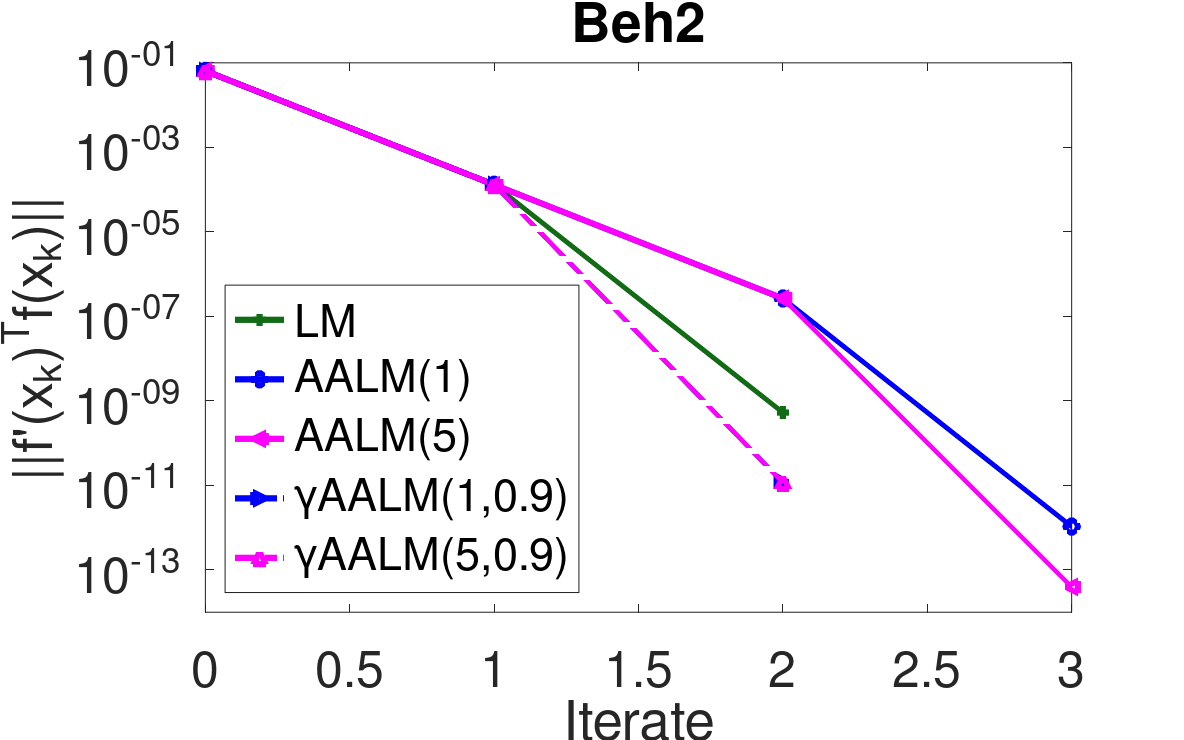}
    \caption{Left: Residual history for test function Beh1. Right: Residual history for test function Beh2. Dash-dotted lines correspond to AALM with $\gamma$-safeguarding. Note that $\gamma$AALM(1,0.9) and $\gamma$AALM(5,0.9) coincide in each of these experiments.} 
    \label{fig:beh1-and-2}
\end{figure}
\begin{figure}[H]
    \centering
\includegraphics[width=0.5\linewidth]{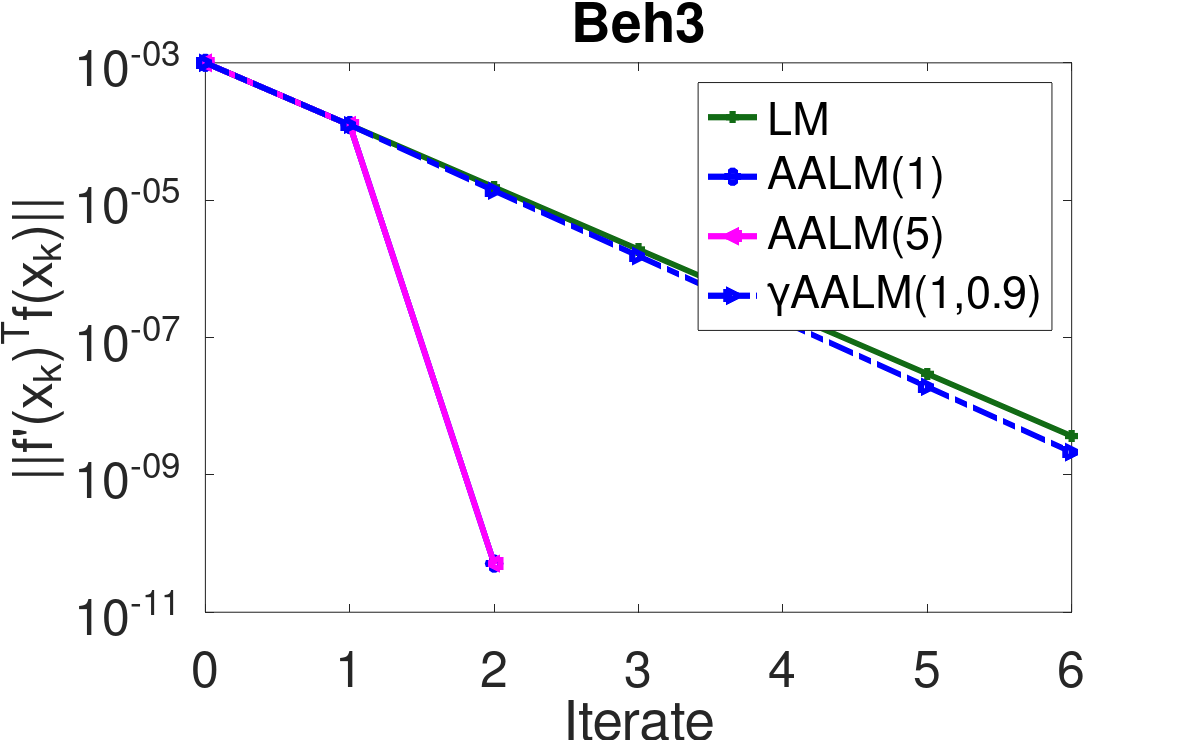}
    \caption{Residual history for test function Beh3. Dash-dotted lines correspond to $\gamma$-safeguarded AALM. As in Figure \ref{fig:beh1-and-2}, $\gamma$AALM(1,0.9) and $\gamma$AALM(5,0.9) coincide. $\gamma$AALM(5,0.9) is therefore omitted.}
    \label{fig:beh3}
\end{figure}
\begin{figure}[H]
    \centering
    \includegraphics[width=0.49\linewidth]{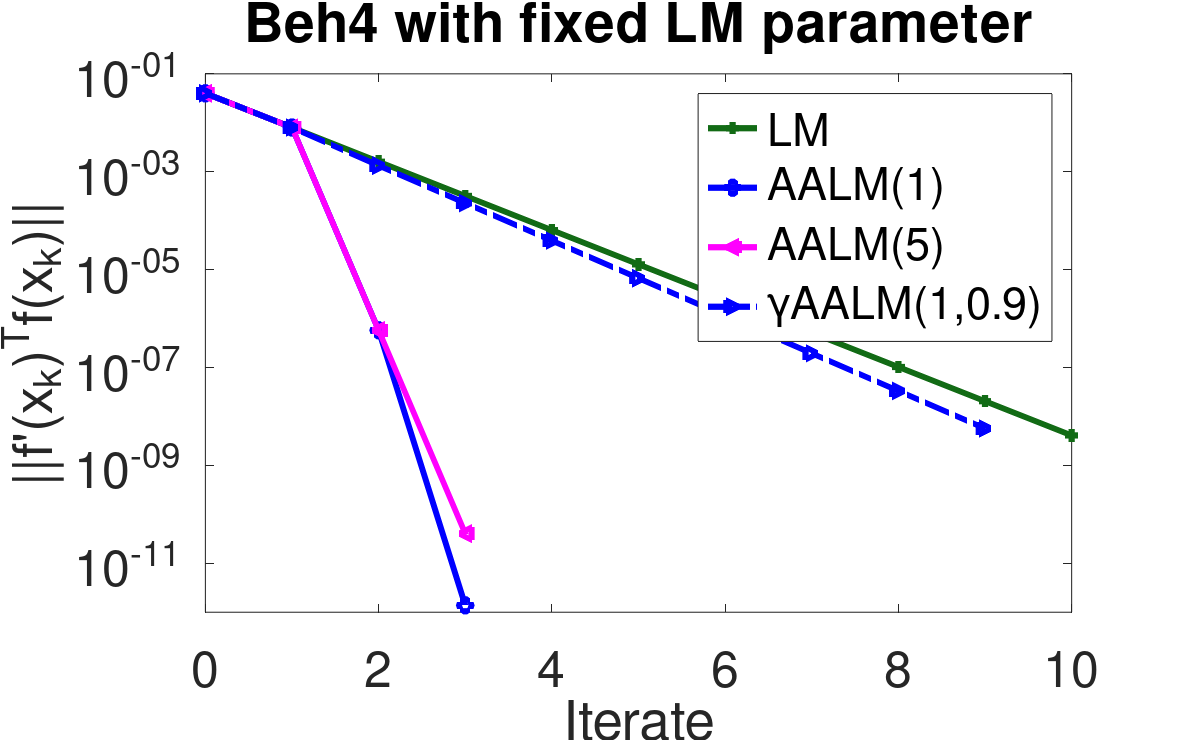}
    \includegraphics[width=0.49\linewidth]{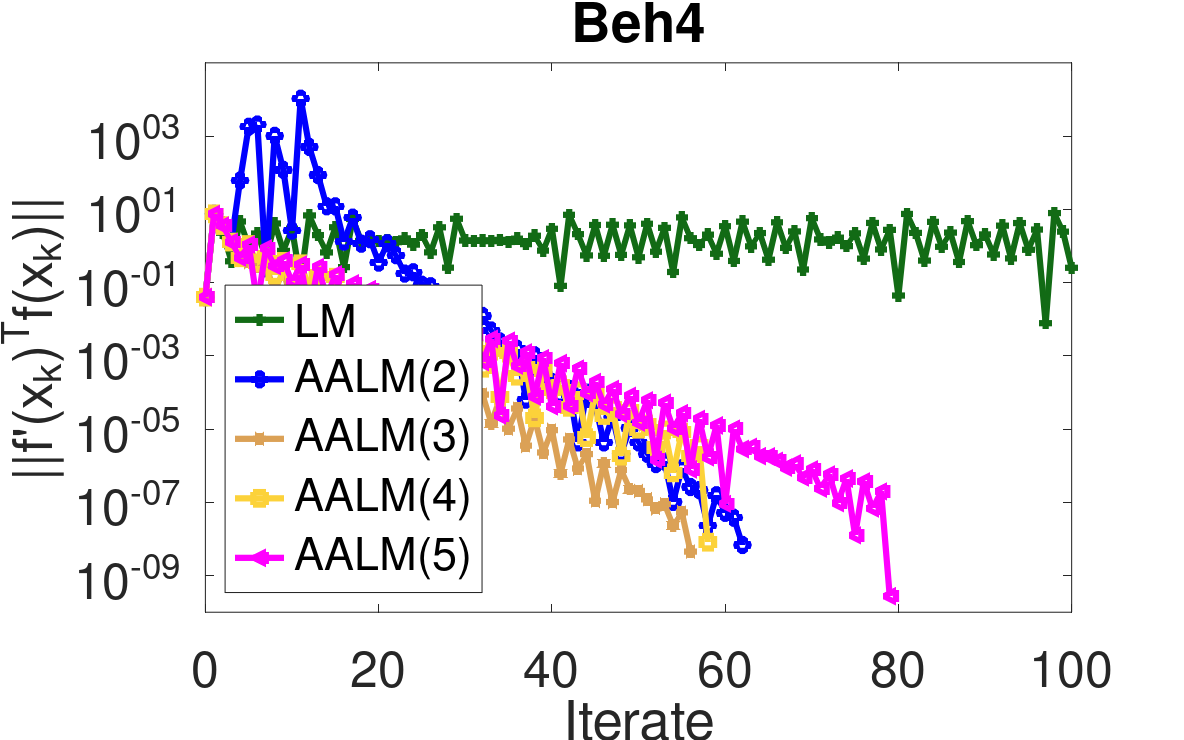}
    \caption{Left: Residual history for test function Beh4 with the LM parameter $\mu_k = 5$ for all $k\geq 0$. Dash-dotted lines correspond to AALM with $\gamma$-safeguarding. $\gamma$AALM(5,0.9) is omitted since it coincides with $\gamma$AALM(1,0.9).} Right: Residual history for test function Beh4 with LM parameter $\mu_k = \|f'(x_k)^Tf(x_k)\|$. 
    \label{fig:beh4}
\end{figure}

\section{Conclusion}

We have presented an analysis of Anderson acceleration applied to the class of nonlinear solvers known as perturbed Newton methods. We proved a local convergence result for Anderson accelerated perturbed Newton methods (AApNMs) with $\gamma$-safeguarding under the assumption of 2-regularity. This result extends previous analysis by the author of Anderson acceleration with algorithmic depth 1 applied to Newton's method for singular problems. A corollary of this analysis provides a novel result for Anderson acceleration applied to the Levenberg-Marquardt (LM) method for singular problems. Further, we showed that $\gamma$-safeguarded Anderson acceleration preserves the local order of convergence of a superlinearly convergent fixed-point iteration, whereas standard Anderson acceleration decreases the order of convergence when applied to such a problem. We demonstrated AApNMs using problems in the literature, and we showed that Anderson accelerated LM (AALM) is viable even when the nonlinear least-squares problem has a non-zero minimum, though further study is needed to fully explain the behavior of AALM applied to such problems when there is a non-unique solution.   

\section*{Acknowledgments}

The author expresses appreciation to Sara Pollock for initial feedback on this project, and to the anonymous referee whose thoughtful comments lead to significant improvements to this article. 

\bibliography{refs}
\bibliographystyle{siam}

\end{document}